\documentclass[12pt]{article}
\usepackage{fullpage}
\usepackage{amsthm}
\usepackage{amsmath}
\usepackage{amssymb}
\usepackage{graphicx}
\usepackage{booktabs}
\usepackage[ruled,vlined,linesnumbered]{algorithm2e}
\usepackage{comment}

\newtheorem{thm}{Theorem}
\newtheorem{lem}{Lemma}
\newtheorem{cor}{Corollary}

\usepackage{enumitem}
\numberwithin{equation}{section}
\newcommand{\regsym}{\textsuperscript{\tiny\textregistered}}
\newcommand{\trademark}{\textsuperscript{\tiny\texttrademark}}
\newcommand{\F}{\mathbb{F}}

\newcommand{\Fq}{\F_q}

\def\Rad{\mathop{\mathrm{Rad}}}
\title{Linear combinations of primitive elements of a finite field}

\author{ Stephen D. Cohen \\
  School of Mathematics and Statistics, \\
  University of Glasgow, Scotland \\
  stephen.cohen@glasgow.ac.uk
\and  
   Tom\'{a}s Oliveira e Silva \\
  Departamento de Electr{\'o}nica, Telecomunica{\c c}{\~o}es e Inform{\'a}tica / IEETA \\
  University of Aveiro, Portugal \\
  tos@ua.pt
  \and
  Nicole Sutherland \\
  Computational Algebra Group, \\
  School of Mathematics and Statistics, \\
  University of Sydney, Australia \\
  nicole.sutherland@sydney.edu.au
\and
  Tim Trudgian\footnote{Supported by Australian Research Council Future Fellowship FT160100094.} \\
School of Physical, Environmental and Mathematical Sciences\\ The University of New South Wales Canberra, Australia \\
  t.trudgian@adfa.edu.au
}

\begin{document}

\maketitle

\begin{abstract}
  \noindent We examine linear sums of primitive roots and their inverses in finite fields. In particular, we refine a result by Li and Han, and show that every $p> 13$ has a pair of primitive roots $a$ and $b$ such that $a+ b$ and $a^{-1} + b^{-1}$ are also primitive roots mod $p$.

\end{abstract}


\section{Introduction} \label{intro}
Let $\mathbb{F}_q$ denote the finite field of order $q$, a power of the prime $p$. The proliferation of primitive elements of $\mathbb{F}_q$ gives rise to many interesting properties. For example, it was proved in \cite{COT} that for any non-zero $\alpha, \beta, \epsilon \in  \mathbb{F}_q$ the equation $
\epsilon = a \alpha + b \beta$
is soluble in primitive elements $a, b$ provided that $q>61$. Since $a$ is primitive if and only if $a^{-1}$, its multiplicative inverse in $\mathbb{F}_q$, is primitive, one may look for linear relations amongst primitive elements and their inverses and, as in the above example, seek a lower bound on $q$ beyond which such relations hold --- this is the purpose of the current paper.

Given
 arbitrary non-zero elements $u, v \in\mathbb{F}_q$, call a pair ($a,b$) of primitive elements of $\mathbb{F}_q$ \emph{$(u,v)$-primitive}
 if  additionally the elements $ua+vb$  and $va^{-1}+ub^{-1}$ are each primitive.  The task is to
  find an asymptotic expression for $N=N(q,u,v)$, defined as  the number of $(u,v)$-primitive pairs  $(a, b)$ in $\mathbb{F}_q$.

In the situation in which  $\mathbb{F}_q$ is a prime field, i.e., $q=p$, this problem was introduced by Li and Han \cite{LiHa}.  In that context,  $a, b$ are considered
to be integers in $I_{p}= \{1,2, \ldots, p-1\}$ with inverses  $a^{-1}, b^{-1} \in I_p$.    Similarly, $u, v$ can be taken to be in $I_p$.
To state the result of \cite{LiHa} we introduce some notation.  For a positive integer $m$ let $\omega(m)$ be the number of distinct prime divisors of $m$ and
$W(m)=2^{\omega(m)}$ be the number of square-free divisors of $m$.
Further, define $\theta(m)$ as $\phi(m)/m$, where $\phi$ is Euler's function, and
$\tau(m)=\prod_{l|m}\left(1-\frac{1}{l-1}+\frac{1}{(l-1)^2}\right)$, where the product is taken over all $\omega(m)$ distinct prime divisors $l$ of $m$.
\begin{thm} [Li--Han]
\label{sheep}
Let $p$ be an odd prime and $n$ any integer in $I_p$.  Set  $\theta =\theta(p-1)$, $\tau= \tau(p-1)$ and $W=W(p-1)$.  Then

\begin{equation}
 \label{goat}
\left|N(p,1,n)-\theta^3\tau\cdot (p-1)^2\right| \leq 5 \theta^4 W^4  p^{3/2}.
\end{equation}

\end{thm}

 Li and Han gave the following as  corollaries to Theorem \ref{sheep}.
\begin{cor}[Li--Han]\label{lamb}
Every sufficiently  large $p$ has primitive roots $a$ and $b$  such that both $a+b$ and $a^{-1}+b^{-1} $ are also primitive. Also,
every sufficiently  large $p$ has primitive roots $a$ and $b$ such that both $a-b$ and $b^{-1}-a^{-1} $ are also primitive.
\end{cor}

We establish an improved estimate for $N(q, u, v)$ in the case of a general finite field.

 \begin{thm}\label{bull} Let $q>2$ be a prime power.  Set $\theta= \theta(q-1), \tau=\tau(q-1), W=W(q-1)$.
 Then, for arbitrary non-zero $u, v \in \mathbb{F}_q$,

 \begin{equation}\label{horse}
\left| N(q,u,v) - \theta^3\tau\cdot(q-1)\  q\right| \leq \theta^4 W^3\cdot(q-1)\sqrt{q}.
  \end{equation}
 \end{thm}

 \smallskip
 The principal improvement in  Theorem \ref{bull}  over Theorem \ref{sheep} is the reduction from $W^4$ to $W^3$ in the error term.
  Its effect can be described as follows.   Let $\mathcal{S}$ be the set of prime powers $q$ such that, for any pair
 of non-zero elements ($u,v$) in $\mathbb{F}_q$, there exists a $(u,v)$-primitive pair in $\mathbb{F}_q$.  Explicit calculations using
  (\ref{goat}) guarantee that all $q$ exceeding $5.7\times10^{364}$ (or with $\omega(q-1)>150$)  are in $\mathcal{S}$.
  On the other hand, using (\ref{horse}), we conclude that all prime powers $q$ exceeding $1.7 \times 10^{84}$ (or with $\omega(q-1) >46$)
  are in $\mathcal{S}$.

 For existence questions
 it is clear that the interest in Theorems \ref{sheep}
and \ref{bull} lies in their lower bounds.   Hence we shall describe a method that, while not delivering an asymptotic estimate,
 establishes a lower bound for $N(q,u,v)$. This yields a non-trivial lower bound applicable to a wider range of prime powers $q$.
    Let  $\Rad(m)$ be the {\em radical} of $m$, i.e., the
    product of the distinct primes dividing a positive integer $m$, and let $\Rad(q-1)$
 be expressed as $kp_1\cdots p_s$
 for some divisor $k$ and distinct primes $p_1, \ldots,p_s$.
 Define $\delta_4=1-4\sum_{i=1}^s\frac{1}{p_i}$.

 \begin{thm}\label{platinum}
Suppose $\delta_4 >0$. Set $\theta=\theta(k), \tau=\tau(k), W=W(k)$. Then
$$ N(q,u,v) \geq \delta_4  \theta^3\cdot(q-1)\{\tau\  q -\theta W^3\sqrt{q}\}.$$
\end{thm}

A consequence of Theorem \ref{platinum} is that  all prime powers $q$
exceeding $6.9\times10^{10}$
are in $\mathcal{S}$.
A stronger conclusion, however,  can be drawn by introducing a subset $\mathcal{T} \subseteq \mathcal{S}$.

Define a single primitive element $a$ to be \emph{$(u,v)$-primitive} if, additionally, $ua+va^{-1}$ is primitive, and define $\mathcal{T}$ as the set of prime powers $q$
such that, for any pair  of non-zero elements ($u,v$) in $\mathbb{F}_q$, there exists a $(u,v)$-primitive element in $\mathbb{F}_q$ (in the above sense).
Easily, if $a$ is a  $(u,v)$-primitive element, then $(a,a^{-1})$ is a $(u,v)$-primitive pair so that, indeed, $\mathcal{T} \subseteq \mathcal{S}$.

For even $q$  the existence of $(1,1)$-primitive elements was the simpler topic\footnote{The harder problem treated in \cite{WaCaFe} and \cite{Co14} concerned the existence of  a $(1,1)$-primitive element in an extension
 field $\mathbb{F}_{q^n}$ which is also normal over the base field $\mathbb{F}_q$.} considered  by Wang, Cao and Feng in \cite{WaCaFe}. Their investigations were completed by Cohen \cite{Co14} --- see also the reference to \cite{Co87} at the end of Section \ref{existence}.
Results on the existence of $(1,1)$-primitive elements in
 $\mathbb{F}_q$  have recently been given by Liao, Li and Pu \cite{LiLiPu}. In this paper
  we use a sieving method and some computation  to establish the 
  following theorem.
 \begin{thm}\label{ox}  Define
 \begin{equation}\label{fridge}
 \begin{split}
 \mathcal{E_T}&=\{2,3,4,5,7,9,11,13,19,25,29,31,37,41,43,49,61,81,97,121,169\},\\
\mathcal{E_S} &= \{2,3, 4,5,7, 13\}.
\end{split}
 \end{equation}
 Then $\mathcal{E_T}$ is the set of prime powers \emph{not} in $\mathcal{T}$ and $\mathcal{E_S}$ is the set of prime powers \emph{not} in $\mathcal{S}$.
 \end{thm}



This generalises and resolves completely the problem posed by Li and Han in \cite{LiHa}. From Theorem \ref{ox} we can easily deduce the following, which resolves completely the `sufficiently large' of Corollary \ref{lamb}.

 \begin{cor} \label{calf} Let $q$ be a prime power.
 \begin{enumerate} [label=\emph{(\roman*)}]
 \item Suppose $q \not \in \{2,3,4,5,7,9,13,25,121\}$. Then there is  a primitive element $a$ in $\mathbb{F}_q$  such that $a+ a^{-1}$ is  primitive.
 \item Suppose $q \not \in \{2,3,4,5,9,13,25,61,121\}$. Then there is  a primitive element $a$ in $\mathbb{F}_q$  such that $a-a^{-1}$ is  primitive.
 \item  Suppose $q \not \in \{2, 3, 4, 5, 7, 13\}$. Then there are primitive elements $a$ and $b$ in $\mathbb{F}_q$  such that both $a+b$ and $a^{-1}+b^{-1} $ are also primitive.
 \item Suppose $q \not \in \{2, 3, 4, 5, 13\}$. Then there are primitive elements $a$ and $b$ in $\mathbb{F}_q$  such that both $a-b$ and $b^{-1}-a^{-1} $ are also primitive.
 \end{enumerate}
 \end{cor}
The outline of this paper is as follows. In \S \ref{prelim} we introduce some notation that, in \S \ref{buffet}, allows us to prove Theorem \ref{bull}. In \S \ref{pairs} we introduce a sieve and prove Theorem \ref{platinum}. In \S \ref{skew} we introduce some asymmetry  and prove Theorem \ref{emerald}, which is sometimes stronger in practice than Theorem \ref{platinum}. In \S \ref{earl} we prove Theorems \ref{lioncub} and \ref{tiger}, which are criteria for membership of $\mathcal{T}$. Finally, in \S \ref{existence} and \S\ref{comp_res} we present theoretical and computational results that prove Theorem~\ref{ox}.

   \section{Preliminaries}\label{prelim}
   To set  Theorem \ref{platinum} and subsequent results in context,  we introduce an extension of the concept of a primitive element in $\mathbb{F}_q$.
     Let $e$ be a divisor of $q-1$.    Then a non-zero element
   $a \in \mathbb{F}_q$ is defined to be {\em $e$-free} if $a=b^d$, where $b \in \mathbb{F}_q$ and $d|e$, implies $d=1$.  This property depends only on $\Rad(e)$.
   In particular, $a$ is primitive if and only if it is $(q-1)$-free.

   Given $e|q-1$, the characteristic function $\lambda_e$ for the subset of $e$-free elements of $\mathbb{F}_q^*$ is expressed in terms of the multiplicative
   characters of $\mathbb{F}_q$
 and is given by
$$ \lambda_e(a) =\theta(e) \sum_{d|e} \frac{\mu(d)}{\phi(d)}\sum_{\chi  \in \Gamma_{d}} \chi(a). $$
Here $\Gamma_{d}$ denotes the set  of  $\phi(d)$ multiplicative characters of  order $d$.
Consistent with dependence only on $\Rad(e)$ is the fact that the only non-zero contributions to
$\lambda_e(a)$  can arise from \emph{square-free} values of $d$: we can assume throughout  that every value of $d$ considered is square-free.
Finally, we generalise  the definition of $\delta_4$ used in Theorem \ref{platinum}.

\begin{equation} \label{delta}
\delta _j = \delta_j(p_1,\ldots,p_s)=1-j\sum_{i=1}^s\frac{1}{p_s}.
\end{equation}
Specifically, in the sequel, we shall employ $\delta_4, \delta_3$ and $\delta_2$.

   \section{Asymptotic estimate for $(u,v)$-primitive pairs}\label{buffet}

In this section we prove Theorem \ref{bull}. Assume that $q$ and non-zero elements $u,v$ of
   $\mathbb{F}_q$ are given.    Writing $\lambda=\lambda_{q-1}$, we have from \S \ref{prelim} that
$$N:=N(q,u,v)= \sum_{a,b \neq 0} \lambda(a)\lambda(b)\lambda(ua+vb)\lambda(va^{-1}+ub^{-1}). $$
   Hence
  \begin{equation} \label{pig}
   N= \theta^4 \sum_{\substack{d_j|q-1,\\j=1,\ldots, 4}} \frac{\mu(d_1)\mu(d_2)\mu(d_3)\mu(d_4)}{\phi(d_1)\phi(d_2)\phi(d_3)\phi(d_4)}\sum_{\substack{\chi_j \in \Gamma_j\\j=1,\ldots,4}}S,
   \end{equation}
   where
   \begin{eqnarray}\label{piglet}
    S&=&  \sum_{a\neq 0}\sum_{b\neq 0} \chi_1(a)\chi_2(b) \chi_3(ua+bv)\chi_4(va^{-1}+ub^{-1}) \nonumber\\
     &=& \sum_{a\neq 0}\sum_{b\neq 0} \chi_1(ab) \chi_2(b)\chi_3(uab+vb)\chi_4(v a^{-1}b^{-1}+ub^{-1}) \nonumber\\
       &=& \sum_{a \neq 0}\sum_{b\neq 0}\chi_1\chi_2\chi_3\chi_4^{-1}(b) \chi_{1}\chi_{4}^{-1}(a) \chi_{3}\chi_{4}(ua+v).
\end{eqnarray}

  If  $\chi_1\chi_2\chi_3\chi_4^{-1}\neq \chi_0$, the principal character, then the sum over $b$ in (\ref{piglet}) is zero, whence $S=0$.
  So in what follows assume  $\chi_1\chi_2\chi_3\chi^{-1}_4= \chi_0$.

   If $\chi_1\chi_4^{-1}=\chi_3\chi_4=\chi_0$, then  $\chi_1=\chi_2=\chi_4=\chi_3^{-1}$ and $S=(q-1)(q-2)$.
   Hence $d_1=d_2=d_3=d_4$ and,    as in \cite[p.\ 7]{LiHa},  the contribution of all such terms in (\ref{pig}) to $N$ is
\begin{equation}\label{piggy1}
  \theta^4 \cdot(q-1)(q-2) \sum_{d|q-1}\frac{\mu^4(d)}{\phi^3(d)} =\theta^3 \tau\cdot(q-1)(q-2).
  \end{equation}

  If $\chi_1\chi_4^{-1}=\chi_0$ (so that $\chi_2\chi_3=\chi_0$) but $\chi_3 \chi_4 \neq \chi_0$, then
  $$S= -(q-1) \sum_{a \neq 0}\chi_3\chi_4(ua+v) =-\chi_3\chi_4(v)(q-1),$$ so that $|S|=q-1$.   Similarly,
  $|S|=q-1$ when $\chi_3\chi_4=\chi_0$ but  $\chi_1\chi_4^{-1}\neq \chi_0$.

  Finally, if $\eta_1=\chi_1\chi_4^{-1}\neq \chi_0$ and $\eta_2 = \chi_3 \chi_4 \neq \chi_0$,
  then
  \[S=(q-1)\sum_{a\neq 0}\eta_1(a) \eta_2(ua+v)= \eta_1\eta_2(v)\eta_1(-1/u)(q-1)J(\eta_1, \eta_2), \]
  where $J$ denotes the Jacobi sum, so that $|S|= (q-1)\sqrt{q}$.

  We now obtain a bound for $|M|$, where $M$ is the sum of terms in (\ref{pig}) corresponding to characters of square-free orders with $\chi_4=\chi_1\chi_2\chi_3$
  excluding those with $\chi_1=\chi_2 =\chi_4=\chi_3^{-1}$ (which were accounted for in (\ref{piggy1})). Thus, we sum over all characters $\chi_1, \chi_2, \chi_3$
  and allow $\chi_4$ to be defined by $\chi_4=\chi_1\chi_2\chi_3$, in which case $d_4$ is the degree of the resulting character $\chi_4$. In general,
  $d_4$ is not determined by $d_1,d_2,d_3$ so we simply use the bound $\phi(d_4) \geq 1$.
    For simplicity, we
  use the bound $|S| \leq (q-1)= (q-1)\sqrt{q}- (q-1)( \sqrt{q}-1)$ whenever $\chi_1=\chi_4$ (so $\chi_2\chi_3=\chi_0$) and include terms with $\chi_1=\chi_2 =\chi_4=\chi_3^{-1}$,
  but the bound $|S| \leq (q-1) \sqrt{q}$, otherwise. Thus

  \begin{equation}\label{piggy2}
  |M| \leq \theta^4 (W^3 \cdot(q-1)\sqrt{q} -|L|),
\end{equation}
where $L$ accounts for the discrepancy in terms with $\chi_1=\chi_4$ so that
$$|L|=\sum_{d_1, d_2|(q-1)}\frac{|\mu(d_1)|\mu^2(d_2)}{\phi(d_1)\phi^2(d_2)}\sum_{\substack { \chi_1\in \Gamma_{d_1}\\ \chi_2 \in \Gamma_{d_2}}}(q-1) (\sqrt{q}-1).$$

Hence
\begin{equation}\label{piggy3}
|L| = (q-1)(\sqrt{q}-1)W\sum_{d|q-1}\frac{\mu(d)^{2}}{\phi(d)}=\frac{W(q-1)(\sqrt{q}-1)}{\theta}.
\end{equation}

Combining (\ref{piggy1}), (\ref{piggy2}), (\ref{piggy3}) we deduce that

\begin{equation}\label{hog}
|N -\theta^3\tau \cdot (q-1)(q-2)| \leq \theta^3\cdot(q-1)\{\theta W^3 \sqrt{q}-W(\sqrt{q}-1)\},
\end{equation}

  The implicit upper bound in (\ref{horse}) is immediate from (\ref{hog}). The lower bound  follows
  from (\ref{hog}) since $2\tau < W\cdot(\sqrt{q}-1)$.

 \begin{cor}\label{cow}
 The prime power $q$ is in $\mathcal{S}$ whenever $q >W^6(q-1)$.
 \begin{proof}By looking at the factors  from each prime $l|q-1$ we see that $\tau(q-1)>\theta(q-1)$.
 \end{proof}
 \end{cor}

 \section{Sieving for $(u,v)$-primitive pairs}\label{pairs}

We now introduce the sieving machinery and prove Theorem \ref{platinum}, which is an improvement on Theorem \ref{bull}.
   As in \S \ref{intro}, write $\Rad(q-1)=kp_1\cdots p_s$, where $p_1, \ldots, p_s$ are the
 \emph{sieving primes}.  For divisors $e_1,\ldots,e_4$ of $q-1$, denote by
 $N(e_1,e_2,e_3,e_4)$ the number of non-zero pairs $a,b \in \mathbb{F}_q$ for which, respectively,
 $a, b,ua+v b,va^{-1}+ub^{-1}$ are $e_1,e_2,e_3,e_4$-free. When $e_1=e_2=e_3=e_4=e$ abbreviate $N(e_1,e_2,e_3,e_4)$ to $N_e$.
 In particular, $N=N_{q-1}$.

 \begin{lem}\label{hen}
We have
 $$ N_{q-1} \geq \sum_{i=1}^s\{N(p_ik,k,k,k)+N(k,p_ik,k,k) +N(k,k,p_ik,k) +N(k,k,k,p_ik)\} -(4s-1)N_k.$$

 Hence, with  $\delta_4$ defined by $(\ref{delta})$,
\begin{equation} \label{goose}
\begin{split}
  N_{q-1} \geq &\sum_{i=1}^s\{[N(p_ik,k,k,k)-\theta(p_i)N_k]+[N(k,p_ik,k,k)-\theta(p_i)N_k]\\ &+[N(k,k,p_ik,k)-\theta(p_i)N_k]
 +[N(k,k,k,p_ik)-\theta(p_i)N_k]\}
 +\delta_4 N_k.
 \end{split}
  \end{equation}
  \end{lem}

As with previous applications of the sieving method, we need an estimate for the various differences appearing in (\ref{goose}).  Somewhat surprisingly, in this instance,
they vanish.
\begin{lem}\label{gold}
For $i=1, \ldots,s$,
\begin{equation*}
N(p_ik,k,k,k)-\theta(p_i)N_k=0.
\end{equation*}
Similarly, the other differences in $ (\ref{goose})$ vanish.
\end{lem}
\begin{proof}
As in (\ref{pig})
\begin{equation} \label{dog}
N(p_ik,k,k,k)-\theta(p_i)N_k=\theta(p_i)\theta^4(k) \sum_{\substack{d_j|k,\\j=1,\ldots, 4}} \frac{\mu(p_id_1)\mu(d_2)\mu(d_3)\mu(d_4)}{\phi(p_id_1)\phi(d_2)\phi(d_3)\phi(d_4)}
   \sum_{ \substack{\chi_1\in \Gamma_{p_id_1}\\\chi_j\in \Gamma_{d_j}\\ j=2,\ldots,4}}S,
 \end{equation}
where $S$ is given by (\ref{piglet}).  Now, in every character sum $S$ appearing in (\ref{dog}), $\chi_1\chi_2\chi_3\chi_4^{-1}$ has degree divisible
by $p_i$ (since this is so for $\chi_1$, but not any of $\chi_2,\chi_3, \chi_4$), whence $S=0$.
\end{proof}
 Since, by Lemmas \ref{hen} and \ref{gold}, we have $N(q,u,v)=N_{q-1}\geq \delta_4N_k$, the argument of Theorem \ref{bull} (based on (\ref{piggy2}), (\ref{piggy3}) and
 (\ref{hog}) but with  $k$ instead of $q-1$) yields Theorem \ref{platinum}.

As a consequence of Theorem \ref{platinum} we deduce an extension  of Corollary \ref{cow}.

\begin{cor}\label{diamond}
Suppose $\delta_4 >0$.
 Then the  prime power $q$ is in $\mathcal{S}$ whenever $q >W^6(k)$.
 \end{cor}
 \begin{proof} As for Corollary \ref{cow}.
 \end{proof}

 Of course, the assumption $\delta_4>0$  is critical for the deduction of  Corollary \ref{diamond}.  Once this holds, unusually (because of Lemma \ref{gold}),
 the criterion does not depend on $\delta_4$.

\section{An asymmetric sieve}\label{skew}

We now obtain a result, in Theorem \ref{emerald}, that is sometimes, though not always, stronger than Theorem \ref{platinum}. We do this by considering some asymmetrical situations in \S \ref{pairs}.
\begin{lem}\label{crab} With notation as in \S $\ref{pairs}$ set $\theta=\theta(k), \tau=\tau(k), W=W(k)$.    Further, write $\theta_{q-1}$ for $\theta(q-1)$
and $N_{k,q-1}$ for $N(k,k,k,q-1)$.
Then

$$ N_{k,q-1}\geq \theta^2\theta_{q-1} \cdot(q-1)\{\tau\  q -\theta W^3\sqrt{q}\}.$$
\end{lem}
\begin{proof} As at (\ref{pig})

\begin{equation}
\label{swine}
 N_{k,q-1}=\theta^3\theta_{q-1} \sum_{\substack{d_1,d_2,d_3|k,\\d_4|q-1}} \frac{\mu(d_1)\mu(d_2)\mu(d_3)\mu(d_4)}{\phi(d_1)\phi(d_2)\phi(d_3)\phi(d_4)}
   \sum_{ \substack{\chi_j\in \Gamma_{d_j}\\ j=1,\ldots,4}}S,
\end{equation}
   where $S$ is given by (\ref{piglet}) and so is zero unless $\chi_4=\chi_1\chi_2\chi_3$.  Now, if $d_4 \nmid k$ then the degree of $\chi_4$
   is not a divisor of $k$ and hence $\chi_4 \neq \chi_1\chi_2\chi_3$, whence $S=0$.  It follows that, in (\ref{swine}), we can restrict $d_4$ to divisors of
   $k$. The lemma then follows as in the proof of Theorem \ref{platinum}.
   \end{proof}
We may take $k= \Rad(q-1)$ in Lemma \ref{crab} to obtain another proof of the lower bound of Theorem \ref{bull}.

   The (obvious) asymmetric version of Lemma \ref{hen} features $\delta_3$ in place of $\delta_4$.
   \begin{lem} We have
   $$ N_{q-1} \geq \sum_{i=1}^s\{N(p_ik,k,k,q-1)+N(k,p_ik,k,q-1) +N(k,k,p_ik,q-1) \} -(3s-1)N_{k,q-1}.$$

 Hence, with $\delta_3$ defined by $(\ref{delta})$
\begin{equation} \label{turkey}
\begin{split}
  N_{q-1} \geq &\sum_{i=1}^s\{[N(p_ik,k,k,q-1)-\theta(p_i)N_{k,q-1}]+[N(k,p_ik,k,q-1)-\theta(p_i)N_{k,q-1}]\\&+[N(k,k,p_ik,q-1)-\theta(p_i)N_{k,q-1}]\}
 +\delta_3 N_{k,q-1}.
 \end{split}
  \end{equation}
  \end{lem}
  The various differences in (\ref{turkey}) do not vanish (cf. Lemma \ref{gold}) but can be usefully bounded.

  \begin{lem} \label{brass}
  For $i=1, \ldots,s$,
$$|N(p_ik,k,k,q)-\theta(p_i)N_{k,q-1}| \leq \frac{1}{p_i}\theta^3\theta_{q-1}W^3 \cdot(q-1)\sqrt{q},$$
where $\theta=\theta(k), W=W(k)$.
A similar bound applies to the other differences in $(\ref{turkey})$.
  \end{lem}
  \begin{proof}  In the expansion of $\Delta=N(p_ik,k,k,q)-\theta(p_i)N_{k,q-1}$ into character sums $S$ analogous to (\ref{dog}) or (\ref{swine}), the degree of the $\chi_1$
   must be  $p_id_1$, where $d_1|k$. But, since $S$ vanishes unless $\chi_4=\chi_1\chi_2\chi_3$, we need only include terms in which the degree of $\chi_4$ similarly is
   $p_id_4$, where $d_4|k$.  Hence
   \begin{equation} \label{pup}
\Delta=\theta(p_i)\theta^3(k)\theta_{q-1} \sum_{\substack{d_j|k,\\j=1,\ldots, 3}} \frac{\mu(p_id_1)\mu(d_2)\mu(d_3)\mu(p_id_4)}{\phi(p_id_1)\phi(d_2)\phi(d_3)\phi(p_id_4)}
   \sum_{ \substack{\ \chi_1\in \Gamma_{p_id_1}\\\chi_2\in \Gamma_{d_2}, \chi_3\in \Gamma_{d_3}\\ \chi_4=\chi_1\chi_2\chi_3}}S,
   \end{equation}
   where $S$ is given by (\ref{piglet}) and the degree of $\chi_4$ is written as $p_id_4$ with $d_4|k$. Since $|S| \leq (q-1)\sqrt{q}$ for each occurrence in (\ref{pup}), and
   $\phi(p_id_4) \geq \phi(p_i)$, it follows that
   $$|\Delta| \leq \theta^3(k)\theta_{q-1} \frac{\theta(p_i)}{{p_i-1}}W^3\cdot(q-1) \sqrt{q}$$
   and the result follows because $\theta(p_{i})/(p_i-1) =1/p_i$.
  \end{proof}

   \begin{thm}\label{emerald}
Suppose $\delta_3 >0$. Set $\theta=\theta(k),  \tau=\tau(k), W=W(k)$. Then
\begin{equation*}
 N(q,u,v) \geq   \theta^2\theta_{q-1}\cdot(q-1)\{\delta_3 \tau\  q -\theta W^3\sqrt{q}\}.
 \end{equation*}
\end{thm}
\begin{proof} Apply the bounds of Lemmas \ref{crab} and \ref{brass} to (\ref{turkey}).  We obtain
$$ N_q \geq \delta_3\theta^2\theta_{q-1} \cdot(q-1)\{\tau\  q -\theta W^3\sqrt{q}\}
-\left\{\sum_{i=1}^s\frac{3}{p_i}\right\}\theta^3\theta_{q-1}W^3 \cdot(q-1)\sqrt{q}.$$
The result follows since $-\sum_{i=1}^s (3/p_i)=1- \delta_3$.
\end{proof}

Generally, Theorem \ref{emerald} gives a better bound than Theorem \ref{platinum} because it allows us to choose more sieving primes, i.e., a larger value of $s$.

 \section{$(u,v)$-primitive elements}\label{earl}
 For given prime power $q$ and non-zero elements $(u, v)$ in $\mathbb{F}_q$ define $M=M(q,u,v)$  as the number  of primitive elements
 in $\mathbb{F}_q$  such that $ua+va^{-1}$ is also primitive.  More generally, for divisors  $e_1, e_2$ of $q-1$, define $M_{e_1,e_2}$ to be the number
 of (non-zero) elements $a \in \mathbb{F}_q$ such that $a$ is $e_1$-free and $ua+va^{-1}$ is $e_2$-free and abbreviate $M_{e,e}$ to $M_e$.
  Then
 \begin{equation} \label{cat}
 M_e= \theta^2\sum_{d_1|e}\sum_{d_2|e}\frac{\mu(d_1)\mu(d_2)}{\phi(d_1)\phi(d_2)}\sum_{ \deg \chi_1=d_1}\sum_{\deg \chi_2=d_2}T,
\end{equation}
where $\theta= \theta(e)$ and
$$T= \sum_{a \in \mathbb{F}_q} \chi_1(a)\chi_2\left(\frac{ua^2+v}{a}\right)=\sum_{a \in \mathbb{F}_q} \chi_1\chi_2^{-1}(a)\chi_2(ua^2+v).$$

If  $\chi_1=\chi_2=\chi_0$, then $T=q-1-\varepsilon$, where $\epsilon$ is the number of zeros of $ua^2+b$ in  $\mathbb{F}_q$. (Here, $\varepsilon$
is 0 or 2 if $q$ is odd and 1 if $q$ is even.)

If $\chi_2=\chi_0$ but $\chi_1 \neq \chi_0$, then $|T|\leq \varepsilon$.

If $\chi_1=\chi_2 \neq \chi_0$, then $|T| \leq \sqrt{q}$.

If $\chi_1 \neq \chi_2$ and $\chi_2 \neq \chi_0$, then $|T| \leq 2\sqrt{q}$.

Hence, from (\ref{cat}),

\begin{equation} \label{kitty}
|M_e-\theta^2\cdot(q-1-\varepsilon)|\leq \theta^2\{2\sqrt{q}A-\sqrt{q}B-(2\sqrt{q}-\varepsilon)C+\sqrt{q}-\varepsilon\}.
\end{equation}
In (\ref{kitty}),
$$A=\sum_{d_1|e}\sum_{d_2|e}\frac{|\mu(d_1)\mu(d_2)|}{\phi(d_1)\phi(d_2)}\sum_{ \deg \chi_1=d_1}\sum_{\deg \chi_2=d_2}1=W^2,$$
where $W=W(e)$. Further,
$$B=\sum_{d|e}\frac{\mu^2(d)}{\phi^2(d)}\sum_{ \deg \chi=d}1= \sum_{d|e}\frac{\mu^2(d)}{\phi(d)}=1/\theta.  $$
Finally,
$$C=\sum_{d|e}\frac{\mu(d)}{\phi(d)}\sum_{ \deg \chi=d}1=W.$$
It follows that
\begin{equation}\label{lion}
|M_e-\theta^2\cdot(q-1-\varepsilon)|\leq \theta^2\left\{2\sqrt{q}\left[W^2-W-\frac{1}{2}\left(\frac{1}{\theta}-1\right)\right]+\varepsilon(W-1)\right\}.
\end{equation}
We may take $e=q-1$ in (\ref{lion}) to deduce an asymptotic expression for $M(q,u,v)$ (see also \cite[Thm.\ 1.3]{LiLiPu}), thereby proving the following theorem.

\begin{thm} \label{lioncub}
Let $q$ be a prime power and set $\theta=\theta(q-1)$ and $W=W(q-1)$. Then
$$\left|M(q,u,v)-\theta^2\cdot(q-1-\varepsilon)\right|\leq \theta^2\left\{2\sqrt{q}\left[W^2-W-\frac{1}{2}\left(\frac{1}{\theta}-1\right)\right]+\varepsilon(W-1)\right\}.$$
\end{thm}

We now introduce the sieve (with the usual notation).  First, here is the analogue of Lemma \ref{hen}.

\begin{lem}\label{cock} Define $\delta_2$ by $(\ref{delta})$.  Then
$$ M\geq \sum_{i=1}^s\{[M_{p_ik,k}-\theta(p_i)M_{k,k}]+[M_{k,p_ik}-\theta(p_i)M_{k,k}]\} +\delta_2 M_{k,k}.$$
 \end{lem}

 \begin{lem} \label{puma} Let $\theta=\theta(k), W=W(k)$.  Then
 \begin{enumerate} [label=\emph{(\roman*)}]
 \item $\displaystyle{|M_{p_ik,k}-\theta(p_i)M_{k,k}| \leq \theta^2\left(1-\frac{1}{p_i}\right)\cdot\{2\sqrt{q}(W^2-W)+\varepsilon W\}}.$
 \item $\displaystyle{ |M_{k,p_ik}-\theta(p_i)M_{k,k}|  \leq 2\left(1-\frac{1}{p_i}\right)\theta^2W^2\cdot \sqrt{q}} .$
 \end{enumerate}
 \end{lem}

\begin{proof}
$$|M_{p_ik,k}-\theta(p_i)M_{k,k}| =\sum_{d_1|k}\sum_{d_2|k}\frac{|\mu(p_id_1)\mu(d_2)|}{\phi(p_i)\phi(d_1)\phi(d_2)}\sum_{ \deg \chi_1=p_id_1}\sum_{\deg \chi_2=d_2}T,$$
and (i)  follows from the bounds on $T$.

Similarly, (ii) holds, since the only character sums involved have $\chi_1 \neq \chi_2$ and $\chi_2 \neq \chi_0$.
\end{proof}
\begin{thm} \label{tiger} Assume $q>3$ is a prime power.
With $\theta=\theta(k), W=W(k)$ (as in Lemma $\ref{puma}$), suppose $\delta_2 >0$.
Then
\begin{equation}\label{leopard}
M > \theta^2\sqrt{q} \left\{\sqrt{q} -2\left(\frac{2s-1}{\delta_2}+2\right)\left[W^2-\frac{W}{2}\left(1-\frac{1}{\sqrt{q}}\right)\right]\right\}.
\end{equation}
Hence, if
\begin{equation} \label{cheetah}
\sqrt{q} >2\left(\frac{2s-1}{\delta_2}+2\right)\left[W^2-\frac{W}{2}\left(1-\frac{1}{\sqrt{q}}\right)\right],
\end{equation}
then $q \in \mathcal{T}$.
\end{thm}
\begin{proof}
Apply Lemma \ref{puma} to the bound of Lemma \ref{cock}.  Observe that
$$\sum_{i=1}^s \left(1- \frac{1}{p_i}\right)=\frac{1}{2}(2s-1+\delta_2).$$
Hence
$$M \geq \delta_2\theta^2\left\{2 \sqrt{q}\left(\frac{2s-1}{\delta_2}+1\right)\left[W^2-\frac{W}{2}\left(1-\frac{\varepsilon}{2\sqrt{q}}\right)\right]+U\right\},$$
where
$$U=(q-1-\varepsilon)-\left\{2\sqrt{q}\left[W^2-W-\frac{1}{2}\left(\frac{1}{\theta}-1\right)\right]+\varepsilon(W-1)\right\} $$
The inequality (\ref{leopard}) follows, since $W\sqrt{q}> \varepsilon W+1 \ (\varepsilon \leq 2)$, certainly for $q>3$.

The criterion (\ref{cheetah}) is then immediate.
\end{proof}

\section{Existence proofs}\label{existence}
In this section we begin to prove Theorem \ref{ox} by demonstrating, using the theorems we have established, that all but finitely many $q$ are members of $\mathcal{T}$and $\mathcal{S}$.   Specifically, we prove that all but  at most 3031 prime powers $q$ are in $\mathcal{T}$ and all but at most 532 values of $q$ are not in $\mathcal{S}$.   Moreover, the possible exceptions could be listed explicitly (although we do not do so).

First, we can show that $q \in \mathcal{T}$
by establishing that the (stronger) sufficient inequality  $q>4W^4(q-1)$ (derived from Theorem \ref{lioncub}) holds whenever $\omega(q-1) \geq 17$,  and the inequality
\[q>  4\left(\frac{{2s-1}}{\delta_2}+2\right)^2W^4(k) \]
(derived from Theorem \ref{tiger}) holds with $s=5$  whenever  $9\leq \omega(q-1)\leq 16$.  We therefore only have to consider those $q$ with $\omega(q-1) \leq 8$.

Now, for each value of $1\leq \omega(q-1) \leq 8$ we find a value of $s\in[1, \omega(q-1) -1]$ such that the right-side of (\ref{cheetah}) is minimised --- call this $q_{\textrm{max}}$. Now $q-1 \geq p_{1} p_{2} \ldots p_{\omega(q-1))}$:= $q_{\textrm{min}}$. We therefore need only check $q\in (q_{\textrm{min}}, q_{\textrm{max}})$.

For example, when $\omega(q-1) = 8$ we choose $s=5$, whence $\delta_{2} \geq 1 - 2(1/7 + 1/11 + 1/13 + 1/17 + 1/19)> 0.1557$ and so $q_{\textrm{max}} < 5.15\cdot 10^{7}$. We also have that $q_{\textrm{min}} = 9,699,691$. We enumerate all prime powers in $(q_{\textrm{min}}, q_{\textrm{max}})$ and select those with $\omega(q-1) = 8$. There are 49 such values, the largest of which is $q= 51,269,791$. For each of these 49 values we now compute the exact value of $\delta_{2}$ for each $s$. For example, for $s=5$ and $q= 51,269,791$ we have $\delta_{2} > 0.387$ --- a considerable improvement. We now look to see whether (\ref{leopard}) holds for these values of $q$. We find that (\ref{leopard}) is true for all but 9 values, the largest of which is $31, 651, 621$. 

We continue in this way, the only deviation from the above example being that for $\omega(q-1) = 1$ we use Theorem \ref{lioncub}. Our results are summarised in Table~\ref{mountain}, which lists, for each value of $\omega(q-1)$, the number of  $q$ for which Theorem \ref{tiger} \emph{fails}  to show that $q \in \mathcal{T}$. Table \ref{mountain} also gives the least and greatest prime and prime power in each category.

 \begin{table}[ht]
 \centering
 \caption{\it{Numbers of primes and prime powers $q$ not shown to be in $\mathcal{T}$.}}
 \label{mountain}
 \medskip
 \begin{tabular}{c c r r |c r r}
   \hline

   $\omega(q-1)$& primes  & least & greatest & prime powers & least & greatest \\
   \hline
   8 & 9 & 13123111 & 31651621 & 0 & - & -\\
   7 & 171 & 870871 & 10840831 & 2 & 2042041& 7447441\\
   6 & 698& 43891 & 2972971& 11& 175561& 1692601 \\
   5 & 951& 2311 & 813121 & 18& 17161 & 776161 \\
   4 & 813& 211 & 102061 & 30 & 841& 63001 \\
   3 & 257& 31 & 9721 & 16& 343 & 2401 \\
   2& 40& 7 & 769 & 9& 16 & 289\\
   1& 3& 3 & 17 &  3 &  4 & 9\\
   \hline

   \hline
 \end{tabular}

 \end{table}
In total, in Table~\ref{mountain}, there are 2942 prime values of $q$ which may not be in $\mathcal{T}$. The prime $2$ is excluded (but clearly $2 \not \in \mathcal{T}$).
 The total of 89 (non-prime) prime powers  comprise 69 prime squares and 20 higher powers.  Unsurprisingly, the latter are powers of small primes as follows: $2^3,2^4,2^5,2^6,$ $2^8,2^{10},$ $2^{12}$, $3^3,3^4,$ $3^5,3^6,5^3,5^4,5^6,7^3,$ $7^4, 11^3, 11^4, 13^3, 13^4, 31^3$.
 Excluding $q=2$, the above leaves a total of 3031 possible prime powers $q$ as candidates for non-membership. Let $\mathcal{C_T}$ be the set of these 3031 candidates. 
We reduce this number substantially in Theorem \ref{final_T}, and, in \S\ref{Hillary} prove Theorem \ref{ox}.

We can pass our list $\mathcal{C_T}$ of 3031 possible exceptions through Theorems \ref{bull}, \ref{platinum} and \ref{emerald}. Note that these test for membership of $\mathcal{S}$. We find that there are only 532 possible prime powers not in $\mathcal{S}$. 



 At this point it is pertinent to add some remarks on the parity of $q$. When $q$ is even, a $(1,1)$-primitive element is the same as a $(1,-1)$-primitive element.  It was proved in \cite{Co14}
  that all fields $\mathbb{F}_{2^n},  n\geq 3$, contain a $(1,1)$-primitive element.  This proof was theoretical, except for the
  values $n=6$ and $12$ when an explicit $(1,1)$-primitive element was given.  In fact, in the preparation
  of \cite{Co14}, the first author overlooked previous work of his \cite{Co87} in which a theoretical proof was given
  (even in these two difficult cases).
  An explanation for the oversight is
  that \cite{Co87} was framed in the notions  of  Garbe \cite{Ga} relating to the order and level of an irreducible polynomial,
  rather than an element of the field.
 In fact, for $q$ even, existence was established in \cite{Co87}, Theorem 5.2.
  In any event, we can assume from now on that $q$   is odd.

\section{Computational results}
\label{comp_res}

In this section we give algorithms and timings for our computations verifying that many
$q$ are in 
$\mathcal{T}$.

Let $w=u^{-1}v$ and $r=a+wa^{-1}$. Then
\[
  ua+va^{-1} = u(a+wa^{-1}) = ur,
\]
and if $u$ and $v$ are non-zero elements of~$\Fq$ then $w$ will also be a non-zero element of~$\Fq$. Thus, to verify
if $q\in\mathcal{T}$ it is sufficient to verify that for all non-zero elements $u$ and $w$ of~$\Fq$,
 there exists a primitive element $a$ of~$\Fq$ such that $ur$ is also a primitive element of~$\Fq$. The
transformation of the original  problem into one with a multiplicative structure allows
discrete logarithms to be used. As primitive elements are easy to characterise using discrete
logarithms, this will give rise to important computational savings.

Let $\gamma$ be a primitive element of~$\Fq$ and let $\log v$ denote the base $\gamma$ discrete logarithm of the
non-zero element $v$ of $\Fq$. Let $p_1,p_2,\ldots,p_{\omega(q-1)}$ be the distinct prime divisors of~$q-1$, and let $R$ be
their product (the radical of $q-1$). If $u$ and $r$ are both non-zero then $ur$ is a primitive element of~$\Fq$
if and only if
\[
  \gcd\bigl( \log u + \log r , q-1 \bigr) = 1,
\]
i.e., if and only if
\[
  \log u \not\equiv -\log r \bmod p_i,\qquad i=1,\ldots,\omega(q-1).
\]
For a given~$w$, it follows that each primitive element $a$ for which $r$ is non-zero takes care of
$\prod_{i=1}^{\omega(q-1)} (p_i-1)$ residue classes of $\log u\bmod R$.
The first result of this setup is Algorithm~\ref{algo_Te}.

\begin{algorithm}[h]
\DontPrintSemicolon
\AlgoDontDisplayBlockMarkers
\SetAlgoNoEnd
\SetAlgoNoLine
\SetKwProg{Proc}{Procedure}{}{}%
\SetKwFunction{checkbetagamma}{check\_beta\_inv\_gamma}%
\SetKwFunction{checkq}{check\_q}%
\SetKwFunction{rad}{rad}%
\SetKw{KwNext}{next}%
\SetKw{KwTo}{in}%
\Proc{\checkq{$q$}}{
    \textit{Construct $\Fq$ and primitive element $\gamma$}\;
    $R \leftarrow$ \rad($q - 1$)\;
    \For{$0 \leq j < q-1$}{
        $v \leftarrow \gamma^j$\;
        \For{$0 \leq k \leq R$}{
            $u \leftarrow \gamma^k$\;
            \For{$l$ in stored\_logs}{
                \If{GCD(k+l, R) = 1}{
                    \KwNext $k$\;
                }
            }
            \For{$1 \leq m < q-1$}{
                \If{GCD(m, R) = 1}{
                    $a \leftarrow \gamma^m$\;
                    $l \leftarrow \log_{\gamma}(a + v a^{-1})$\;
                    Store $l$ in stored\_logs\;
                    \If{GCD(k+l, R) = 1}{
                        \KwNext $k$\;
                    }
                }
            }
            \If{$m = q-1$}{
                FAIL\;
            }
        }
    }
}
\caption{Check whether $q \in \mathcal{T}$\label{algo_Te}}
\end{algorithm}

To maximise efficiency in Algorithm~\ref{algo_Te} we store the $\log$s we
have computed as well as the elements which have already been determined to
be primitive so we can first check through our list of stored primitive
elements $a$ and only generate more primitive elements as needed.

We can write $\mathcal{E_T}$ (defined in (\ref{fridge}))
as
$$\{2, 3, 4, 5, 9\} \cup \{ 7, 11, 13, 19, 25, 29, 37, 41, 49, 81, 97\} \cup
\{31, 43, 61, 121, 169\}$$ according to $\omega(q-1)$.
It has been checked by running Algorithm~\ref{algo_Te}
using Magma~\cite{magma223} that
$\mathcal{E_T} \cap \mathcal{T} = \emptyset$ and
that $q \in \mathcal{T}$ for all $q \in \mathcal{C_T} \setminus \mathcal{E_T}, \omega(q-1) < 7$.
In Table~\ref{timings_T} we provide total timings for these checks for all
$q \in \mathcal{C_T}$, grouped by $\omega(q-1)$, on a
2.6GHz Intel\regsym{} Xeon\regsym{} E5-2670 or similar machine. 
Note that checking that the 140 $q$ with $\omega(q-1)=6$ of the 532 $q$ which are not known to be in $\mathcal{S}$ are in $\mathcal{T} \subset \mathcal{S}$ took only 178 days
and each such $q$ could be checked in less than 3.4 days.

\begin{table}[h]
\begin{centering}
\begin{tabular}{|c|c|c|c|c|c|c|c|c|c|}
\hline
$\omega(q-1)$ & 1 & 2 & 3 & 4 & 5 & 6\\
Number of $q$ checked & 6 & 49 & 273 & 843 & 969 & 709 \\ 
Time & 1.87 & 2.7s & 591s & 1.9 days & 98.9 days & 20.003 years \\ 
\hline
\end{tabular}\caption{Total timings for checking whether $q \in \mathcal{T}$ for $q \in \mathcal{C_T}$.}\label{timings_T}
\end{centering}
\end{table}

We found it efficient to store primitive elements $a$ and check
$u (a + v a^{-1})$ for primitivity with the stored $a$ first before
generating more primitive elements. Note that we do not need to check
all $(u, v)$ pairs since $u a + v a^{-1} = v a^{-1} + u (a^{-1})^{-1}$
so if there is a $(u, v)$-primitive element there is also a $(v, u)$-primitive
element. However, this improvement is made redundant by factoring out $u$ 
and iterating through only $R$ many.

We noticed that if $a$ is a $(u, v)$-primitive element then it is also a
$(u', v')$-primitive element when $(v'-v) = -a^2(u - u')$.
Unfortunately, these observations  did not improve the
efficiency of our algorithms.


\begin{algorithm}[h]
\DontPrintSemicolon
\AlgoDontDisplayBlockMarkers
\SetAlgoNoEnd
\SetAlgoNoLine
\SetKwProg{Proc}{Procedure}{}{}%
\SetKwFunction{checkbetagamma}{check\_beta\_inv\_gamma}%
\SetKwFunction{checkq}{check\_q}%
\SetKwFunction{rad}{rad}%
\SetKw{KwNext}{next}%
\SetKw{KwTo}{in}%
\Proc{\checkq{$q$}}{
    \textit{Construct $\Fq$ and primitive element $\gamma$}\;
    \For{$0 \leq k < q-1$}{
        $u \leftarrow \gamma^k$\;
	\For{$0 \leq l < q-1$}{
	    $v \leftarrow \gamma^l$\;
	    \For{$0 \leq m < q-1$}{
		\If{GCD(m, q-1) = 1}{
		    $a \leftarrow \gamma^m$\;
		    \For{$0 \leq n < q-1$}{
			\If{GCD(n, q-1) = 1}{
			    $b \leftarrow \gamma^n$\;
			    \If{$u a + v b$ and $v a^{-1} + u b^{-1}$ are primitive}{
				\KwNext $l$\;
			    }
			}
		    }
		}
	    }
	    \If{$m = q-1$}{
		FAIL\;
	    }
        }
    }
}
\caption{Check whether $q \in \mathcal{S}$\label{algo_S}}
\end{algorithm}

We also checked using Algorithm~\ref{algo_S}
whether $q \in \mathcal{S}$ for $q \in \mathcal{E_T}$ and found
that only $2, 3, 4, 5, 7, 13 \notin \mathcal{S}$.
The computations for these
checks took about 1 second using Magma on a
3.4GHz Intel\regsym{} Core\trademark{} i7-3770 or similar machine. 

Finally, we deduce Corollary~\ref{calf} by checking which $(u, v)$ caused failures in
Algorithms~\ref{algo_Te} and~\ref{algo_S}.

We summarise our results from this section in the following theorem.
\begin{thm}
\label{final_T}
All
prime powers $q$ with $\omega(q-1) < 7$ and $q\notin \mathcal{E_T} $ are in $\mathcal{T}$ 
and so also in $\mathcal{S}$. There are at most $182$ values of $q\not\in \mathcal{E_T}$ not in $\mathcal{T}$, the largest of which is $31,651,621$. 
\end{thm}




\section{Improved algorithm to check whether $q\in\mathcal{T}$}
\label{Hillary}
We now introduce a new algorithm to handle the 182 possible exceptions annunciated in Theorem~\ref{final_T}. 

Let $L=\{\,0,1,\ldots,R-1\,\}$ be a complete set of residues modulo~$R$. For $i=1,\ldots,\omega(q-1)$, let
$L'_i=\{\,0,1,\ldots,p_i-1\,\}$ be a complete set of residues modulo~$p_i$, and let
\[
  L_{r,i} = \{\;l \,:\;l\in L \,\wedge\, l\not\equiv -\log r \bmod p_i\;\}, \quad  L'_{r,i} = \{\;l \,:\;l\in L'_i \,\wedge\, l\not\equiv -\log r \bmod p_i\;\}.
\]
Finally, let
\[
  L_r = \bigcap_{i=1}^{\omega(q-1)} L_{r,i}.
\]
By construction, the condition $\gcd\bigl( \log u + \log r , q-1 \bigr) = 1$ is equivalent to the condition
$(\log u \bmod R) \in L_r$. Furthermore, using the Chinese remainder theorem, the number $|L_r|$ of elements
of the set $L_r$ is given by
\begin{equation}
  |L_r| = \prod_{i=1}^{\omega(q-1)} |L'_{r,i}|.
  \label{e:setSize}
\end{equation}

The improved strategy used to check if $q\in\mathcal{T}$ is as follows. For each non-zero value of $w$, use distinct
primitive elements $a_1$, $a_2$, $\ldots$, to construct the corresponding sets $L_{r_1}$, $L_{r_2}$, $\ldots$,
stopping when either the list of primitive elements is exhausted, in which case $q\not\in\mathcal{T}$, or when
the union of these sets is $L$, in which case all non-zero values of $u$ have been covered and so the next non-zero
$w$ needs to be tried. When the $w$ values have been exhausted we conclude that $q\in\mathcal{T}$.

In an actual computer program, sets are usually implemented as arrays of bits\footnote{Each bit of the array
indicates if the corresponding element belongs or does not belong to the set.}, with unions and intersections being bitwise
logical \underline{or} or logical \underline{and} operations. In the present case the $L_r$ sets have $R$~bits, so
using an array of bits to represent them adds a factor of $R$ to the execution time of the program. It turns out that
using the inclusion-exclusion principle~\cite[Chapter~XVI]{HW78} to count the number of elements of a union of sets
gives rise to a considerably faster program. For example
\[
  |L_{r_1} \cup L_{r_2}|=|L_{r_1}|+|L_{r_2}|-|L_{r_1} \cap L_{r_2}|,
\]
so $|L_{r_1} \cup L_{r_2}|$ can be computed by evaluating $|L_{r_1}|$ and $|L_{r_2}|$ using~(\ref{e:setSize}), and
by evaluating the remaining term using
\[
  |L_{r_1} \cap L_{r_2}| = \prod_{i=1}^{\omega(q-1)} |L'_{r_1,i} \cap L'_{r_2,i}|.
\]
All three computations can be done using only the $L'_{r,i}$ sets. As these sets are quite small --- for those $q$, summarily described in Table~\ref{mountain}, with $q\geq 371,281$ that were not tested by the
methods of \S\ref{comp_res} the largest $p_i$ is only $89$ --- they should be implemented as bit arrays. As
these bit arrays can be stored in two 64-bit computer words, counting the number of elements of each one of them
can be done efficiently using the population count instruction available on modern Intel/AMD 64-bit processors.

In general, to count the number of elements in the union of the $n$ sets $L_{r_k}$, $k=1,\ldots,n$, by applying
recursively the inclusion-exclusion principle it is necessary to compute $2^n$~terms. Fortunately, in the present
case most of these terms turn out to be zero, because for a small $p_i$ the intersection of several $L'_{r_k,i}$
sets has a good chance to be the empty set. Nonetheless, to avoid an uncontrolled explosion of the number of
terms as more values of $r$ are considered, the following strategy was used to
accept/reject values of~$r$:
\begin{itemize}
  \item the first $10$ non-zero values of $r$ are always accepted;
  \item the remaining non-zero values of $r$ are accepted only if they lead to a $3/4$ reduction in the number of residue classes that are still not covered.
\end{itemize}
This fast but aggressive strategy failed in a very small percentage of cases (less than $0.003$\% for
$q=31,651,621$). When it failed the same procedure was tried again with the factor $3/4$ replaced by $4/5$. As this never failed for our list of values of~$q$, even more relaxed parameters (more initial
values of $r$ always accepted, larger factors) were not needed.

Denote by $B'_{r,i}$ the bit array of $p_i$ bits corresponding to the set $L'_{r,i}$. In a computer program the
set $L_{r_k}$ can be efficiently represented by the tuple $B'_r=(1,B'_{r,1},\ldots,B'_{r,\omega(q-1)})$, where the
initial $1$ represents the inclusion-exclusion generation number. Intersections of sets can be represented in the
same way, with the generation number reflecting the number of intersections performed. (To apply
inclusion-exclusion it is only necessary to keep track of the parity of the number of intersections.) As
mentioned before, intersecting two sets amounts to performing bitwise \underline{and} operations of the corresponding
$B'_{\cdot,i}$ bit arrays, which, given the small size of these arrays, can be done very quickly on contemporary
processors. In Algorithm~\ref{super_algo_T} the variable $B$ is a list of tuples that represent the non-empty sets
used in the inclusion-exclusion formula. The number of residues classes not yet covered by the union of the
$L_r$ sets, denoted by $|B|$, is given by
\[
  |B|=R+\sum_{B' \in B} (-1)^{\mathrm{generation}(B')}\,|B'|.
\]

These considerations give rise to Algorithm~\ref{super_algo_T} (in step 2, the list of primitive elements can be
constructed so that $a_{\phi(q-1)+1-k}$ is the inverse of $a_k$; that simplifies the computation of the values of
$r$.)

\begin{algorithm}[h]
\DontPrintSemicolon
\AlgoDontDisplayBlockMarkers
\SetAlgoNoEnd
\SetAlgoNoLine
\SetKwProg{Proc}{Procedure}{}{}%
\SetKwFunction{checkq}{check\_q}%
\SetKwFunction{checkw}{check\_w}%
\Proc{\checkq{$q$}}{
    \textit{Construct $\Fq$ and list $a_1,\ldots,a_{\phi(q-1)}$ of the primitive elements of $\Fq$}\;
    \For{each non-zero element $w$ of $\Fq$}{
        \If{\checkw{$w,10,\frac34$} returns 0}{
            \If{\checkw{$w,10,\frac45$} returns 0}{
                \If{\checkw{$w,12,\frac56$} returns 0}{
                    \If{\checkw{$w,\phi(q-1),1$} returns 0}{
		        FAIL\;
	            }
	        }
            }
        }
    }
}
\SetKwProg{Func}{Function}{}{}%
\SetKwFunction{checkw}{check\_w}%
\Func{\checkw{$w,nc,f$}}{
  $c \leftarrow 0$, $B \leftarrow \emptyset$\;
  \For{$1 \leq k \leq \phi(q-1)$}{
    $r \leftarrow a_k+wa^{-1}_k$\;
    \If{$r \not= 0$}{
      $c \leftarrow c+1$\;
      \textit{Compute $B'_r$}\;
      \textit{Intersect $B'_r$ with all sets stored in $B$ and store the non-empty ones in $X$}\;
      \textit{Append $B$ and $B'_r$ to $X$}\;
      \If{$c\leq nc$ or if $|X|\leq f|A|$}
      {
        $A \leftarrow X$\;
        \If{$|A|=0$}{
          \textit{return $1$}\;
        }
      }
    }
  }
  \textit{return $0$}\;
}
\caption{Check whether $q\in \mathcal{T}$}\label{super_algo_T}
\end{algorithm}

Algorithm~\ref{super_algo_T} gave rise to three optimised computer programs, written in the C programming language.
One that dealt with $q$ prime and $\max_i p_i < 64$, another that dealt with $q$ prime and $64 < \max_i p_i < 128$,
and a third one that dealt with $q$ a prime square and $\max_i < 64$. The largest case, $q=p=31,651,621$, was
confirmed to belong to $\mathcal{T}$ in one week on a single core of a $4.4$GHz i7-4790K Intel processor. All
exceptional values of $q$ not dealt with by the methods of \S\ref{comp_res} were confirmed to belong to
$\mathcal{T}$ in about four weeks of computer time (one week of real time, given that the i7-4790K processor has
four cores). These computations were double-checked on a separate machine.

Figure~\ref{execution_times} presents data for all cases that were tested by the first program. It
suggests that the execution time of the program is approximately proportional to $q$ and depends in a non-linear
way on $\omega(q-1)$. The same phenomenon occurs for the other two programs.

\begin{figure}[hpb]
  \centering
  \includegraphics[width=\textwidth]{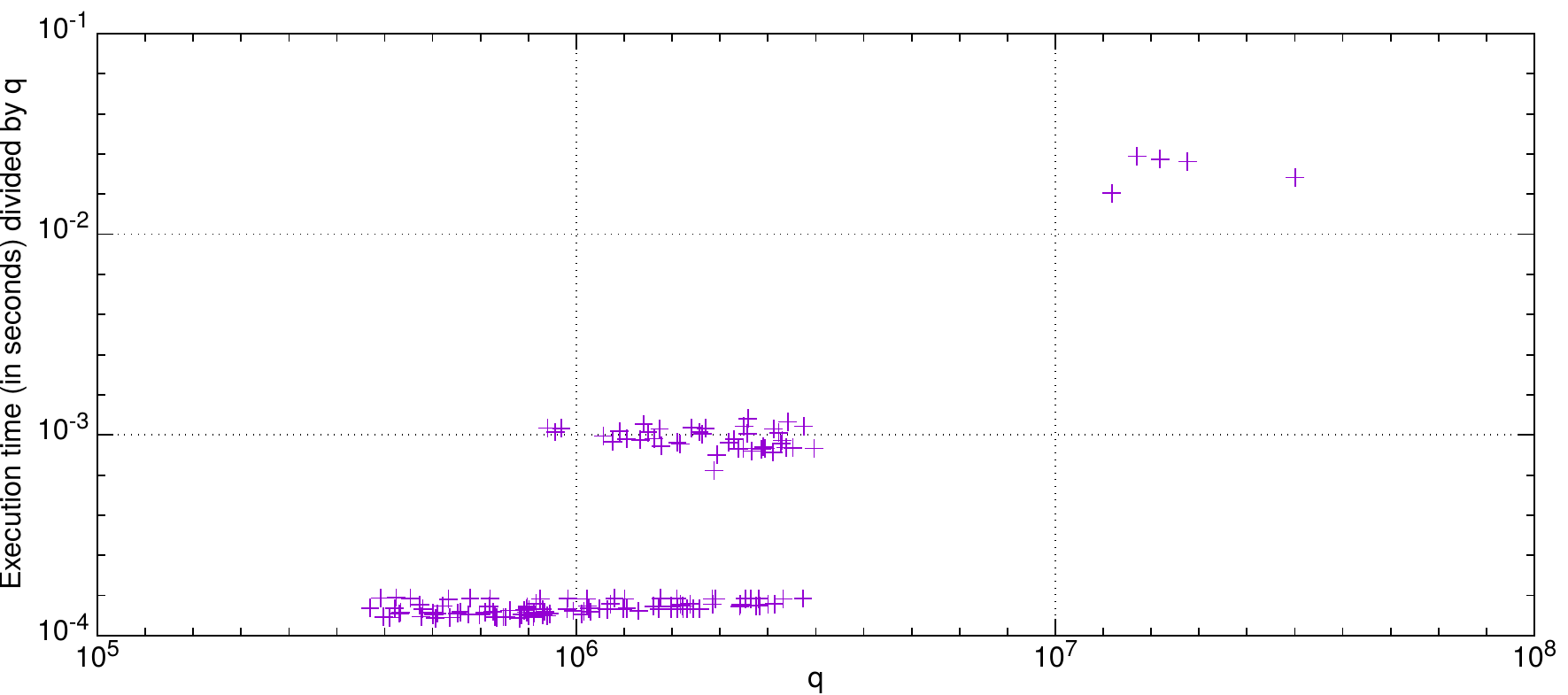}
  \caption{Execution times (in seconds) divided by $q$ versus $q$; the bottom data points correspond to values of
    $q$ for which $\omega(q-1)=6$, those in the middle correspond to $\omega(q-1)=7$ and those on top to
    $\omega(q-1)=8$.}%
  \label{execution_times}%
\end{figure}

\section*{Acknowledgements}
The third author acknowledges the Sydney Informatics Hub and the University of Sydney's high
performance computing cluster Artemis for providing the high performance computing
resources that contributed to the results in \S\ref{comp_res}.

\end{document}